\documentclass[11pt,leqno]{article}
\usepackage{graphicx, amsfonts, amsthm, amsxtra, verbatim, multicol}
\usepackage[mathscr]{euscript}
\title{ On the non-persistence of Hamiltonian identity cycles }
\author{{\sc  L. Gavrilov$^{\ (1)}$, H. Movasati$^{\ (2)}$ and I. Nakai$^{\ (3)}$}}

\begin{document}
\maketitle

\date{}
\begin{center}
\noindent {\small{$^{\ (1)}$ Institut de Math\'{e}matiques de Toulouse, UMR 5219\\
 Universit\'{e}  de Toulouse,  31062 Toulouse,  France  }\\
$^{\ (2)}$ Instituto de Matem\'atica Pura e Aplicada (IMPA) \\
Estrada Dona Castorina 110,  22460-320,
Rio de Janeiro, Brasil \\
$^{\ (3)}$ Department of Mathematics, Ochanomizu University\\
2-1-1 Otsuka, Bunkyo-ku,
Tokyo 112-8610, Japan.
}
\end{center}

\begin{abstract}
We study the leading term of the holonomy map of a perturbed plane polynomial
Hamiltonian foliation. The non-vanishing of this term implies the
non-persistence of the corresponding Hamiltonian identity cycle. We prove that
this does happen for generic perturbations and cycles, as well for cycles
which are commutators in Hamiltonian foliations of degree two. Our approach
relies on the Chen's theory of iterated path integrals which we briefly resume.
\end{abstract}

\newtheorem{theo}{Theorem}
\newtheorem{lem}{Lemma}
\newtheorem{exam}{Example}
\newtheorem{coro}{Corollary}
\newtheorem{defi}{Definition}
\newtheorem{axio}{I}

\newtheorem{prob}{Problem}
\newtheorem{lemm}{Lemma}
\newtheorem{prop}{Proposition}
\newtheorem{rem}{Remark}
\newtheorem{conj}{Conjecture}
\def\End{{\mathrm End}}              
\def\hol{{\mathrm Hol}}
\def\sing{{\mathrm Sing}}            
\def\cha{{\mathrm char}}             
\def\Gal{{\mathrm Gal}}              
\def\jacob{{\mathrm jacob}}          
\newcommand\Pn[1]{\mathbb{P}^{#1}}   
\def\Z{\mathbb{Z}}                 
\def\ZZ{\mathbb{Z}}               
\def\Q{\mathbb{Q}}                   
\def\C{\mathbb{C}}                   
\def\R{\mathbb{R}}                   
\def\N{\mathbb{N}}                   
\def\A{\mathbb{A}}                   
\def\uhp{{\mathbb H}}                
\newcommand\ep[1]{e^{\frac{2\pi i}{#1}}}
\def\Mat{{\mathrm Mat}}              
\newcommand{\mat}[4]{
     \begin{pmatrix}
            #1 & #2 \\
            #3 & #4
       \end{pmatrix}
    }                                
\newcommand{\matt}[2]{
     \begin{pmatrix}                 
            #1   \\
            #2
       \end{pmatrix}
    }
\def\ker{{\mathrm ker}}              
\def\cl{{\mathrm cl}}                
\def\dR{{\mathrm dR}}                

\def\hc{{\mathsf H}}                 
\def\Hb{{\cal H}}                    
\def\GL{{\mathrm GL}}                
\def\pedo{{\cal P}}                  
\def\PP{\tilde{\cal P}}              
\def\cm {{\cal C}}                   
\def\K{{\mathbb K}}                  
\def\F{{\cal F}}                     
\def\M{{\cal M}}
\def\RR{{\cal R}}
\newcommand\Hi[1]{\mathbb{P}^{#1}_\infty}
\def\pt{\mathbb{C}[t]}               
\def\W{{\cal W}}                     
\def\Af{{\cal A}}                    
\def\gr{{\rm gr}}                
\def\ring{{\sf R}}
\def\BM{{\sf H}}
\def\OO{{\cal O}}
\def\rank{{\rm rank }}

\def\ii{{\sf i}}
\def\k{\mathsf{k}}
\def\K{\mathsf{K}}
\def\span{{\rm Span}}
\section{Introduction}

Let  $\mathcal{F}^n$ be a degree $n$ polynomial foliation   on the plane
$\C^2$.
 A cycle on a leaf is a non-zero free homotopy class of closed loops on this leaf.
 Let $\delta$ be a closed non-contractible loop on a leaf and
denote the holonomy map associated to $\delta$ by $h_\delta$. The cycle
represented by $\delta$ is said to be an \emph{identity cycle} provided that
$h_\delta$ is the identity map, and \emph{Hamiltonian identity cycle}, provided
that $\mathcal{F}^n = \{ df=0\}$. The present paper is motivated by the following:
\begin{conj}\rm
\label{conj1}
A generic polynomial foliation on  $\C^2$ can not have identity cycles.
\end{conj}
which follows on its turn from the following more general:
\begin{conj}\rm
\label{conj2}\cite[p.4]{ilpreface}, (D.V. Anosov, early 1960s)
For a generic polynomial foliation by analytic curves on $\C^k$, all leaves are topological discs except for a
countable number of topological cylinders.
\end{conj}
Conjecture \ref{conj1} has its origin in the  the famous Petrovskii-Landis
paper \cite{pe55}, see the comments of Ilyashenko in \cite{il96}.
We consider a restricted version of Conjecture \ref{conj1}:
\begin{conj}\rm
\label{conj3} Let $\delta$ be a Hamiltonian identity cycle. There is a
sufficiently small degree $n$ perturbation of $\mathcal{F}^n$ which either
destroys  $\delta$, or makes it a limit cycle.
\end{conj}
For $n>2$ Conjecture \ref{conj3} is a Theorem as proved by Ilyashenko and
Pyartli \cite[Theorem 2]{ilpy}. In the present paper we propose a different
approach to Conjecture \ref{conj3} based on a well known  integral formula for
the dominant term of the asymptotic expansion of the perturbed holonomy map.
Namely, let
$$
h^\varepsilon_\delta(t)= t + \varepsilon^k M_k(t) + o(\varepsilon^k)
$$
be the asymptotic expansion of the holonomy map $h^\varepsilon_\delta$ associated to the perturbed foliation
$$
df + \varepsilon \omega = 0
$$
and to a cycle of the non-perturbed foliation $\{df=0\}$ (here $\omega$ is a polynomial form). Clearly
Conjecture \ref{conj3} would follow from
 $h^\varepsilon_\delta\neq id$ which on its turn is a consequence of $M_k\not = 0$. The function $M_k$,
 the so called Poincar\'{e}-Pontryagin-Melnikov function, has an integral representation in terms of iterated path
 integrals of length at most $k$, see \cite{ga06}. Let $(F_t^n)_{n\geq 1}$ be the lower central series of the
 fundamental group $F_t$ of the fiber $f^{-1}(t)$. The properties of iterated integrals imply that if $\delta
 \in F_t^k$ then $M_i=0$ for $i=1,2,..., k-1$. The Poincar\'{e}-Pontryagin-Melnikov function $M_k$ is therefore a
 kind of "linearization" of the holonomy map with respect to the deformation parameter $\varepsilon$. We might
 expect that, at least for generic $\omega$, the function $M_k$ is non zero. This claim, if true, would be
 stronger than the claim of Conjecture \ref{conj3}. We shall show, however,  that for $k\geq
 5$ there exist cycles $\delta\in F_t^k$ such that for every polynomial one-form $\omega$ of degree at most $n$,
 holds $M_k=0$ (Example \ref{ex}).

The present paper is devoted to the study of the Poincar\'{e}-Pontryagin-Melnikov
 function $M_k$. Our first result is that for generic $f$, $\omega$ and for
 almost all identity cycles $\delta\in F_t^k/F_t^{k+1}$, the function $M_k$ is
 not identically zero (Theorem \ref{th4}). We investigate then in more details the
 case $n=2$ (the fibers $f^{-1}(t)$ are elliptic curves). We prove that if
 $\delta \in F_t^2/F_t^{3}$ (e.g. $\delta=aba^{-1}b^{-1}$ is a commutator),
 then $M_2\neq 0$ for almost all degree $n$ one-forms $\omega$ (Theorem
 \ref{elliptic}). In particular the identity cycle $\delta$ is destroyed.

The proof of the above results relies in an essential way on the Chen's theory
of interated path integral. We formulate and prove the relevant results in
\S \ref{mainsection}, which can be read independently. The main results
here are Theorem \ref{main} and Theorem \ref{prop1} (the so called $\pi_1$ de
Rham theorem). These results are not completely new, see the paper of R. Hain
\cite{hain86}. Our presentation, as well the proofs are different, and we hope
simpler, as based on a classical combinatorial theorem of Ree \cite{re58}. To
the end of this Introduction we resume \S \ref{mainsection}.

Let $\Gamma= \bar{\Gamma}\setminus S$, where $ \bar{\Gamma}$ is a compact
Riemann surface and $S\subset  \bar{\Gamma}$ a non-empty finite set of points.
Each holomorphic one-form $\omega$ in $\Gamma$ defines a linear map $\int\omega \in
{\rm Hom}_\Z(H_1(\Gamma,\Z),\C)$ by integration:
$$
\int\omega: H_1(\Gamma,\Z)\rightarrow \C: \delta \mapsto \int_\delta\omega
$$
and he classical de Rham theorem stipulates that $ {\rm Hom}_\Z(H_1(\Gamma,\Z),\C)$
is generated by such linear maps. In the context of the present paper
$\Gamma=f^{-1}(t)$ is a leaf of the non-perturbed foliation $\{df=0\}$ and
$\int_\delta\omega$ is the first Poincar\'{e}-Pontryagin-Melnikov function $M_1$.

The Chen's de Rham theorem generalizes  the de Rham theorem as follows. The
fundamental group $F=\pi_1(\Gamma,*)$  is free  finitely generated.
 For $A,B \subset F$ we denote by $(A,B)$ the
subgroup of $F$ generated by commutators $(a,b)=aba^{-1}b^{-1},  a\in A,b \in B$.
 Define by induction the free abelian  subgroups
$F^k= (F,F^{k-1})$, $F^1= F$. The Chen's de Rham theorem, e.g. \cite{hain}, claims that for each $k$
the vector space ${\rm  Hom}_\Z(F^k/F^{k+1},\C)$
is generated by iterated integrals of length $k$. As $F^1/F^{2}=
H_1(\Gamma,\Z)$ then in the case $k=1$ we get the usual de Rham theorem. In the
context of the present paper, the Poincar\'{e}-Pontryagin-Melnikov function $M_k$
defines an element of ${\rm Hom}_\Z(F^k/F^{k+1},\C)$.

The purpose of \S \ref{mainsection} is to prove a more precise version of
this "$\pi_1$ de Rham theorem" by constructing explicitly the space
${\rm Hom}_\Z(F^k/F^{k+1},\C)$ in terms of iterated integrals along Lie elements of
length $k$ (see Theorem \ref{main}). The construction is purely combinatorial,
the only analytic ingredient in the proof being the usual de Rham theorem.
Therefore we can formulate the result in algebraic terms as follows:

Let $X=\{x_1,x_2,...,x_m\}$ be a set
(which does not represent
necessarily loops and one-forms) and $\k\subset \K$  two  fields of characteristic zero.
Denote by ${\rm Ass}_X$ the graded $\k$-algebra of
associative but non-commutative polynomials in variables $x_1,x_2,\dots,x_m$ and by $L_X \subset
{\rm Ass}_X$ the graded Lie algebra generated by $x_1,\dots,x_m$.  Let also
$F_X$ be a free group generated by the elements of $X$.
By an iterated integral we mean
any  map
$$
\begin{array}{rcl}
\int:  F_X \times {\rm Ass}_X &\rightarrow& \K\\
     (\delta,\omega)&\mapsto& \int_{\delta}\omega
\end{array}
$$
which is $\k$-linear in ${\rm Ass}_X$ and satisfies the four axioms given in the next section. It will follow
from these axioms that $\int$ induces
 a well-defined map:
\begin{equation}
\label{toul}  F_X^k/ F_X^{k+1} \times {\rm Ass}_X^k \to \K,\ \ k=1,2,\ldots
\end{equation}
which is $\Z$-linear in the first coordinate and $\k$-linear in the second
coordinate. Of course an example of such a map is the usual path integrals of
Chen mentioned above. The $\pi_1$ de Rham theorem can be formulated as follows:
If the vector space of $\Z$-linear maps
$$
\label{h1}
\begin{array}{rcl}
\int\omega:  F_X/ F_X ^2&\rightarrow& \K, \mbox{  where  }\omega \in L_X^1 \\
     \delta&\mapsto& \int_{\delta}\omega
\end{array}
$$
is isomorphic to ${\rm Hom}_\Z( F_X/ F_X^2,\K)$, then the vector space of $\Z$-linear maps
$$
\begin{array}{rcl}
\int\omega:  F_X^k/F_X^{k+1} &\rightarrow& \K , \mbox{  where  } \omega \in L_X^k \\
     \delta&\mapsto& \int_{\delta}\omega
\end{array}
$$
is isomorphic to
 ${\rm Hom}_\Z(F_X^k/F_X^{k+1},\K)$ for all $k=1,2,\dots$.  We also prove that $\int \omega,\
 \omega \in {\rm Ass}_X^k$ is identically zero if and only if $\omega$ is a linear combination of
 shuffle elements.

{\it Acknowledgments.} Part of this paper was written while the authors were
visiting subsequently  IMPA of Rio de Janeiro (Brasil), the Ochanomizu
University, Tokyo (Japan) and the University of Toulouse (France). They are
obliged for the hospitality.

\section{Iterated path integrals}
\label{noandde}
Let $\k$ be a field. All modules and algebras are taken over $\k$, unless stated otherwise. For a set
$X=\{x_1,x_2,...,x_m\}$, let ${\rm Ass}_X$ be the graded free associative algebra on $X$.
Its elements are the non-commutative polynomials in $x_i$ with coefficients in $\k$. Define a Lie bracket in ${\rm Ass}_X$ by
$[x,y]=xy-yx$, and let $L_X\subset {\rm Ass}_X$ be the graded free Lie algebra on $X$.
 Thus, for instance, $x_1, [x_1,x_2], [[x_1,x_2],x_3]$ belong to $L_X$ but not
$1$, $x_1x_2$, $x_1x_2x_3$. ${\rm Ass}_X$ is the universal  algebra of the Lie algebra $L_X$,
see \cite{se64} .
The graded piece $L_X ^1$ is
just the $\k$-vector space generated by $x_1,x_2,...,x_m$, $L_X ^2$ is the $\k$-vector space
generated by $[x_i,x_j]$ etc.
Each
element of $L_X$ is called a Lie element. We denote by ${\rm Ass}_X^k$ (resp. $L_X^k$) the homogeneous component of degree
$k$ of ${\rm Ass}_X$ (resp. $L_X$).

Let $F_X$ be the free group on $X$, its elements are the words in the letters
$x_i$ and their formal inverses $x_i^{-1}$. For $x,y\in F_X$ we define the
commutator $(x,y)= xyx^{-1}y^{-1}$. For $A,B \subset F_X$ we denote by $(A,B)$
the subgroup of $F_X$ generated by commutators $(a,b),  a\in A,b \in B$.
Consider the lower central series $F_X^n$ of $F_X$, where $F_X^n=
(F_X,F_X^{n-1})$, $F_X^1= F_X$. The associated graded $\Z$-Lie algebra is given
by
\begin{equation}\label{graded}
    \gr F_X = \sum_{n=1}^\infty \gr^n F_X,\quad \gr^n F_X = F_X^{n}/F_X^{n+1},
\end{equation}
$$
  [xF^{i+1}_X,yF_X^{j+1}]=(x,y)F^{i+j+1}_X.
$$
The canonical map $X\to \gr^1F_X$ which send $x_i$ to $x_i$ induces an isomorphism of Lie algebras
\begin{equation}
\label{phimap}
\phi: L_X \to (\gr F_X)\otimes_\Z\k
\end{equation}
(e.g.\cite{se64} Theorem 6.1).
For two words $\omega_{1}\cdots \omega_{r}$, $\omega_{{r+1}} \cdots \omega_{{r+s}}$ define the shuffle product
$\omega_{1}\cdots \omega_{r} * \omega_{{r+1}} \cdots \omega_{{r+s}}$ to be  the sum of all words of length $r+s$
that are permutations of $\omega_{1}\cdots \omega_{r}  \omega_{{r+1}} \cdots \omega_{{r+s}}$ such that both
$\omega_{1}\cdots \omega_{r}$ and $\omega_{{r+1}} \cdots \omega_{{r+s}}$ appear in their original order, e.g.
$$
\omega_1\omega_2*\omega_3= \omega_1\omega_2\omega_3+ \omega_1\omega_3\omega_2+ \omega_3\omega_1\omega_2 .
$$
Let $\K$ be a field extension of the field $\k$.
\begin{defi}\rm
 \label{integral}
An
iterated integral is a map

\begin{equation}\label{iterated}
\begin{array}{rcl}
\int:  F_X \times {\rm Ass}_X &\rightarrow& \K\\
     (\delta,\omega)&\mapsto& \int_{\delta}\omega
\end{array}
\end{equation}
which is $\k$-linear in the second variable and  satisfies the four axioms:
\begin{description}
\item[A1]
For every non-commutative polynomial $\omega \in {\rm Ass}_X$ and $\delta\in F_X$,
   $
 \int_{1}\omega  \in \k
 $
 is the constant term of $\omega$ and
 $
 \int_{\delta}1=1
 $
 for all $\delta\in F_X$.
 We use the convention
 $\omega_1\omega_2\cdots\omega_r=1$ for $r=0$.
 \item[A2]
For $\alpha,\beta\in F_X$ and $\omega_1,\omega_2,\ldots\omega_r\in {\rm Ass}^1_X$
$$
\int_{\alpha\beta}\omega_1\cdots\omega_r=\sum_{i=0}^r\int_{\alpha}
\omega_1\cdots \omega_i\int_{\beta}\omega_{i+1}\cdots\omega_r.
$$
 \item[A3]
For $\alpha\in F_X$ and $\omega_1,\omega_2,\ldots\omega_r\in {\rm Ass}^1_X$
$$
\int_{\alpha^{-1}} \omega_1\omega_2\cdots\omega_r=(-1)^r\int_{\alpha}\omega_r\cdots\omega_1.
$$
 \item[A4]
For $\alpha\in F_X$ and $\omega_1,\omega_2,\ldots\omega_{r+s}\in {\rm Ass}^1_X$ we have
\begin{equation}
\label{7mar06} \int_{\alpha} \omega_{1}\cdots
\omega_{r}\int_{\alpha}\omega_{{r+1}}\cdots\omega_{{r+s}}=\int_{\alpha} \omega_{1}\cdots \omega_{r} *
\omega_{{r+1}} \cdots \omega_{{r+s}},
\end{equation}
where
$$
\omega_{1}\cdots \omega_{r} * \omega_{{r+1}} \cdots \omega_{{r+s}}= \sum
\omega_{k_1}\omega_{k_2}\cdots\omega_{k_{r+s}}
$$
is the shuffle product of $\omega_{1}\cdots \omega_{r} $ and $ \omega_{{r+1}} \cdots \omega_{{r+s}}$.
 \end{description}
 \end{defi}
Let $\{\delta_1,\delta_2,\dots,\delta_m\}$ be a set which generates $F_X$ freely and
$\{\omega_1,\omega_2,\dots,\omega_m\}$ be  a basis of the $\k$-vector
space  ${\rm Ass}_X^1$.
 \textbf{A1}, \textbf{A2}, \textbf{A3} imply that every iterated integral
can be written as a polynomial in
\begin{equation}
\label{nakaisensei} \int_{\delta_j} \omega_{i_1}{\omega_{i_1}}\cdots \omega_{i_r},\ j=1,2,\ldots,m,\
i_1,i_2,\ldots,i_r\in\{1,2,\ldots, m\}.
\end{equation}
Therefore by \textbf{A4}   the map $(\ref{iterated})$ defines an iterated integral if and only if the numbers
$$
 \int_{\delta_j}\omega_{i_1}\omega_{i_2}\cdots \omega_{i_r}=a^j(i_1,\ldots, i_r)\in \k
 $$
satisfy the "shuffle relations"

 \begin{equation}\label{shuffle}
a^j(i_1,\ldots,i_r)a^j(i_{r+1},\ldots,i_{r+s})=\sum a^j(k_1,\ldots,k_{r+s}),
\end{equation}
 where $(k_1,\ldots,k_{r+s})$
 runs through all shuffles of $(i_1,\ldots, i_r)$ and $(i_{r+1},\ldots, i_{r+s})$.
The existence of such numbers $a^j(i_1,\ldots,i_r)$ is, however, not obvious.

 \begin{exam}\rm
 Let $\{\delta_1,\delta_2,\dots,\delta_m\}$ be a set which generates $F_X$ freely and
$\{\omega_1,\omega_2,\dots,\omega_m\}$ be  a basis of the $\k$-vector
space  ${\rm Ass}_X^1$.
 We set
\begin{equation}
\label{6july06} a^j(i_1,\ldots,i_n) = \int_{\delta_j}\omega_{i_1}\omega_{i_2}\cdots \omega_{i_n}=0 \hbox{ if at
least one of $\omega_{i_s}$ is not $\omega_j$},
\end{equation}
$$
a^j(j,\ldots,j)=\int_{\delta_j}\omega_j^n=\frac{1}{n!} .
$$
The verification that the numbers $a^j(i_1,\ldots,i_n)$ define an iterated integral is straightforward.
This iterated integral can be interpreted as Chen's iterated integral in the following way: We take  the
Ceyley  diagrams (see \cite{nakai}) which is the
topological space $Y:=\cup_{i=0}^{m-1}\Z^{i}\times\R\times\Z^{m-i-1}\subset \R^m$. Each element $\delta$ of the
free  group $F_X$ is represented by a path  $\tilde \delta$ in $Y$ with the starting point $0\in\R^n$ and the
end point in some element of $\Z^m$. Each element $\omega$ of ${\rm Ass}_X$ can be interpreted as a differential form
$\tilde \omega$ substituting $dy_i$ by $x_i$, where $(y_1,y_2,\ldots, y_m)$ is the coordinate system in $\R^m$.
Now, $\int_{\delta}\omega$ is the classical Chen's iterated integral $\int_{\bar\delta}\tilde\omega$.
\end{exam}

\begin{exam}\rm
As in the Introduction, let $\Gamma$ be  a punctured Riemann surface with free finitely generated fundamental group
$F_X=\pi_1(\Gamma,*)$.
Let  $\{\omega_1,\omega_2,\ldots,\omega_m\}$ be a collection of holomorphic
one-forms on $\Gamma$ such that their classes in the de Rham cohomology of $\Gamma$ form
a basis.
The Chen's iterated
integral
$$
a^j(i_1,\ldots,i_n) = \int_{\delta_j}\omega_{i_1}\omega_{i_2}\cdots \omega_{i_n},
$$
where $\delta_1,\delta_2,\ldots,\delta_m$ generates $F_X$ freely,
satisfies the shuffle relations (\ref{shuffle}), see \cite{hain1,hain}, and hence defines an iterated integral
in the sense of Definition \ref{integral} (in order to follow the terminology used in this text we may
identify $\omega_i$ with $x_i$).
\end{exam}
\section{The $\pi_1$ de Rham theorem}
\label{mainsection}
Let $\alpha,\beta \in F_X$ and $\omega\in {\rm Ass}_X^1=L_X^1$. Then \textbf{ A3} implies that every iterated integral
(\ref{integral}) satisfies
$$
\int_{\alpha\beta} \omega = \int_{\alpha} \omega + \int_{\beta} \omega, \int_{\alpha\beta\alpha^{-1}} \omega=
\int_{\alpha} \omega, \int_{(\alpha,\beta)} \omega = 0 .
$$
More generally,
\begin{equation}\label{k1}
\forall \alpha\in F^{n+1}_X, \omega\in {\rm Ass}_X^{m}, \mbox{  such that  } m\leq n \mbox{   holds true  }
\int_{\alpha} \omega = 0.
\end{equation} Therefore the iterated integral $\int$ induces a map
\begin{equation}\label{iterated1}
\begin{array}{rcl}
\int:  \gr^kF_X \times {\rm Ass}_X^k &\rightarrow& \C \\
     (\delta,\omega)&\mapsto& \int_{\delta}\omega
\end{array}
\end{equation}
which is $\Z$-linear in the first argument and $\k$-linear in the second argument.
Suppose that $\omega \in {\rm Ass}_X^k$ is a shuffle product of $\omega_1, \omega_2$:
$$\omega = \omega_1 * \omega_2,
\omega_1\in {\rm Ass}_X^{k_1}, \omega_2\in {\rm Ass}_X^{k_2}, k_1+k_2=k
$$
and $\alpha \in \gr^k F_X$. Then \textbf{A4} and (\ref{k1}) imply
$$
\int_{\alpha} \omega = 0 .
$$
Thus the vector space $S^k_X$ generated by  shuffle products
$$
S^k_X= \span\{ \omega_1 * \omega_2 : \omega_1\in {\rm Ass}_X^{k_1}, \omega_2\in {\rm Ass}_X^{k_2}, k_1+k_2=k\}
$$
is in the kernel of the bilinear map (\ref{iterated1}).
We are ready to formulate the $\pi_1$ de Rham theorem:
\begin{theo}
\label{main}
Let $\int$ be an iterated integral in the sense of Definition  \ref{integral}.
\begin{description}
    \item[(a)]
If the induced bilinear map
\begin{equation}\label{pi11}
\begin{array}{rcl}
\int:  \gr^k F_X \times L_X^k &\rightarrow& \K\\
     (\delta,\omega)&\mapsto& \int_{\delta}\omega
\end{array}
\end{equation}
is non-degenerate for $k=1$,  then then it is non-degenerate for all
    $k\in \N$.
    \item[(b)] Let $\omega \in {\rm Ass}^k_X$. The linear map
    \begin{equation}\label{pi12}
\begin{array}{rcl}
\int\omega :  \gr^k F_X  &\rightarrow& \K \\
     \delta&\mapsto& \int_{\delta}\omega
\end{array}
\end{equation}
is the zero map, if and only if $\omega \in S_X^k$.
\end{description}
\end{theo}
In ${\rm Ass}_X$ we consider the canonical $\k$-bilinear symmetric product given by
 $$
\langle x_{i_1}x_{i_2}\cdots x_{i_r},  x_{j_1}x_{j_2}\cdots x_{j_s}\rangle= \left \{
\begin{array}{ll}
1 & i_1=j_1,\ldots, i_r=j_r, r=s \\
0 & \hbox{ otherwise }
\end{array}
\right.
$$
Before proving Theorem \ref{main} we shall need the following auxiliary result:
\begin{theo}
\label{prop1} $S_X^k$ and $L_X^k$ are orthogonal vector subspaces of ${\rm Ass}_X^k$ and for all $k\in \N$
$$
{\rm Ass}_X^k=L_X^k\oplus S_X^k.
$$
\end{theo}
This is a geometric reformulation of the following classical result of Ree \cite[Theorem 2.2]{re58}.
\begin{theo}
\label{ree}
 A polynomial
\begin{equation}\label{ree1}
\omega=\sum_{n>0}\sum a(i_1,i_2,\ldots,i_n)x_{i_1}x_{i_2}\cdots x_{i_n} \in {\rm Ass}_X
\end{equation}
is a Lie element (i.e. belongs to $L_X$) if and only if for all $(i_1,\ldots, i_r)$ and $(j_1,\ldots,j_s)$ we
have:
\begin{equation}\label{ree2}
\sum a(k_1,k_2,\ldots,k_{r+s})=0
\end{equation}
where $(k_1,\ldots, k_{r+s})$ runs through all shuffles of $(i_1,\ldots, i_r)$ and $(j_1,\ldots, j_s)$.
\end{theo}
\begin{proof}[Proof of the equivalence of Theorem \ref{prop1} and Theorem \ref{ree}.]
Without loss of generality we can assume that $\omega$ in (\ref{ree1}) is homogeneous
of degree $n$.
The proof
follows from the equality:
$$
\langle \omega, x_{i_1}x_{i_2}\cdots x_{i_r}*x_{j_1}x_{j_2}\cdots x_{j_s}\rangle=
\sum a(k_1,k_2,\ldots,k_{r+s})
$$
where
$x_{i_1}x_{i_2}\cdots x_{i_r}*x_{j_1}x_{j_2}\cdots x_{j_s},\ r+s=n$ is a shuffle
element and $(k_1,\ldots, k_{r+s})$ runs through all shuffles of $(i_1,\ldots, i_r)$ and $(j_1,\ldots, j_s)$.
Note that we can formulate Theorem \ref{prop1} in the following way: The polynomial $\omega$
is a Lie element if and only if it is orthogonal to all shuffles
$x_{i_1}x_{i_2}\cdots x_{i_r}*x_{j_1}x_{j_2}\cdots x_{j_s},\ r+s=n$.
\end{proof}
%


%

To prove Theorem \ref{main} we note that as
 the map (\ref{pi11}) is non-degenerate for $k=1$, then it can be "diagonalized" as follows:
We fix a set of generators $\{\delta_1,\delta_2,\ldots,\delta_m\}$ for $F_X$ and  find
$\Omega=\{\omega_1,\dots,\omega_m\}\subset {\rm Ass}^1_X$ such that
$\int_{\delta_i}\omega_j=1$ if $i=j$ and $=0$ otherwise. Now, ${\rm Ass}_X$ is freely
generated by $\Omega$.
Therefore we shall suppose, without loss of generality, that
\begin{equation}\label{product}
\int_{x_i} x_j = \langle x_i,x_j\rangle = 0 \mbox{ if } i\neq j \mbox{  and  } 0 \mbox{  otherwise}.
\end{equation}

The formula (\ref{product}) generalizes as follows:
\begin{prop}
\label{prop2} We have
 \begin{equation}
 \label{phi}
 \int_{\delta}{\omega}=\langle \omega, \phi^{-1}\delta\rangle,\
 \forall \omega\in {\rm Ass}^k_X,\ \delta\in \gr^kF_X.
 \end{equation}
 where $\phi$ is the isomorphism (\ref{phimap}).
\end{prop}
\begin{proof}
The proof is by induction on $k$. For $k=1$ it follows from (\ref{product}). If $\delta=(a,b)$, where $a\in
F^p_X, b\in F^q_X$, $p+q=k>1$, then \textbf{A2} implies
\begin{eqnarray*}
  \int_{(a,b)} x_{j_1}x_{j_2}\cdots x_{j_k} &=& \int_a x_{j_1}\cdots x_{j_p}\int_b x_{j_{p+1}}\cdots x_{j_k}\\
  & &- \int_b
x_{j_1}\cdots x_{j_q} \int_a x_{j_{q+1}}\cdots x_{j_k} \\
   &=& \langle\phi^{-1}(a), x_{j_1}\cdots x_{j_p}\rangle \langle\phi^{-1}(b), x_{j_{p+1}}\cdots x_{j_k}\rangle\\
   & & - \langle\phi^{-1}(b),
x_{j_1}\cdots x_{j_q} \rangle\langle\phi^{-1}(a) x_{j_{q+1}}\cdots x_{j_k}\rangle \\
   &=& \langle\phi^{-1}(a)\phi^{-1}(b), x_{j_1}\cdots x_{j_p} x_{j_{p+1}}\cdots x_{j_k}\rangle\\
   & & - \langle\phi^{-1}(b)\phi^{-1}(a),
x_{j_1}\cdots x_{j_q}  x_{j_{q+1}}\cdots x_{j_k}\rangle  \\
  &=& \langle\phi^{-1}(a)\phi^{-1}(b)-\phi^{-1}(b)\phi^{-1}(a),x_{j_1}x_{j_2}\cdots x_{j_k}\rangle\\
  &  =& \langle\phi^{-1}((a,b)), x_{j_1}x_{j_2}\cdots x_{j_k}\rangle.
\end{eqnarray*}
\end{proof}
\begin{rem}\rm
For an arbitrary iterated integral map, the same proof  implies that for every $\delta\in \gr^k F_X$ such
that
$$
\phi^{-1}(\delta) = \sum a(i_1,i_2,\ldots,i_k)x_{i_1}x_{i_2}\cdots x_{i_k} \in \gr^k L_X
$$
holds
\begin{equation}\label{calcul}
\int_\delta \omega_{j_1}\omega_{j_2}\cdots \omega_{j_k}= \sum_{i_1,i_2,\ldots,i_k}
a(i_1,i_2,\ldots,i_k)\int_{x_{i_1}}\omega_{j_1} \int_{x_{i_2}}\omega_{j_2}\dots\int_{x_{i_k}}\omega_{j_k}
\end{equation}
\end{rem}

\begin{proof}[Proof of Theorem \ref{main}.] The part (b) follows from Theorem \ref{prop1} and Proposition \ref{prop2}.
Let $\delta\in \gr^k F_X$ and
$$
\phi^{-1}(\delta) = \sum a(i_1,i_2,\ldots,i_k)x_{i_1}x_{i_2}\cdots x_{i_k} \in \gr^k L_X
$$
where in the above sum each non-commutative monomial $x_{i_1}x_{i_2}\cdots x_{i_k}$ is repeated only once. Then
$\delta\neq 0$ if and only if
$$\langle\phi^{-1}(\delta),\phi^{-1}(\delta)\rangle = \sum |a(i_1,i_2,\ldots,i_k)|^2\neq 0.$$
Therefore for every $\delta\in \gr^k F_X$ there exists $\omega\in L_X^k$, namely $\omega = \phi^{-1}(\delta)$,
such that
$$
\int_\delta \omega \neq 0 .
$$
\end{proof}
Here we have used strongly the fact that the characteristic of $\k$ is zero.
\begin{rem}\rm
\label{27june}
A canonical basis of the $\k$-vector spaces $L_X^k\cong \gr^k F_X$ is given by
basic commutators (see for instance \cite{ha59, se64}).
By definition of $\langle\cdot,\cdot \rangle$ if the number of some
$x_i, i=1,2,\ldots,m$ used in two basic commutators $\omega_1$ and $\omega_2$ are
different then $\langle \omega_1, \omega_2\rangle =0$. The basic commutators of
 weight $1$ and $2$ are dual to each other with respect to the bilinear map
 $\langle\cdot,\cdot\rangle$.
 However, this is not the case for weight $3$ and the number of generators $m$ bigger than $2$.
 For instance, we have
 $$
 \langle [y,[x,z]], [z,[x,y]]\rangle=2
 $$
 For $m=2$ the basic commutators of weight $3$ (resp. $4$) are orthogonal  to each other.
 In  $m=2$ and $r=5$ the orthogonality fails.
 There are two couples in which the  number of $x$ is equal to 2 (resp. 3). In fact we have
 $$
\langle [y,[x,[x,[x,y]]]], [[x,y], [x,[x,y]]]\rangle =-28,\
\langle  [y,[y,[x,[x,y]]]],[[x,y], [y,[x,y]]]\rangle=-14
 $$
 The computation of basic commutators and scalar products in the free Lie algebra $L_X$ are implemented in the symbolic algebra system {\sc Axiom}, see \cite{lambe} and
the webpage of the second name author
 \end{rem}
\section{Picard-Fuchs equations and iterated path integrals }
\label{picardfuchs}
Let $f$ be a polynomial of degree $d$ in two variables $x,y$
and suppose, for simplicity, that the highest order homogeneous piece $g$ of $f$ is non-degenerate (has an isolated critical point). Let also $C\subset \C$ be the set of the critical
values of $f$.  The cohomology fiber bundle
$\cup_{t\in \C\backslash C} H^1(\{f=t\},\C)$ and the corresponding Gauss-Manin connection,
for which flat sections are generated by sections with images in   $\cup_{ t\in \C\backslash C}
H^1(\{f=t\},\Z)$, is encoded
in the global Brieskorn module
$H=\frac{\Omega^1}{df\wedge \Omega^0+d\Omega^0}$, where $\Omega^i$ is the set of
polynomial $i$-forms in $\C^2$ as follows. We note first that
$H$ is a $\C[t]$-module ( $t.[\omega]=[f\omega]$) generated freely by

\begin{equation}
\label{bases}
\omega_{i}:=x^{i_1}y^{i_2}(xdy-ydx), i = (i_1,i_2) \in I
\end{equation}
where  $\{x^{i_1}y^{i_2},\  i\in I\}$
 is a basis of monomials for the $\C$-vector space
$\frac{\C[x,y]}{\langle f_x,f_y\rangle}$,   see
\cite{gav, ho06-6}.
 In particular the rank of $H$ equals the dimension of $H^1(\{f=t\},\Z)$ for generic $t$.
The Gauss-Manin connection
of the family of curves  $f(x,y)=t,\ t\in\C$ with respect to the parameter $t$ becomes an operator
$':H\to \frac{1}{\Delta}H$ which satisfies the Leibniz rule,
where $\Delta=\Delta(t)$ is the discriminant of the polynomial
$f(x,y)-t$. In the basis
$\omega=(\omega_i)_{i\in I}$
(written in a column) it is of the form
\begin{equation}
\label{18}
\omega'=\frac{1}{\Delta}A\omega,
\end{equation}
where $A$ is a $m\times m$ matrix with entries in $\C[t]$ and $m=\# I$.

The above construction has a natural generalization
based on the   $\pi_1$ de Rham theorem (Theorem \ref{main}), which we discuss briefly.
 Let
 $$F_t = \pi_1(\{f=t\},*),\ t\in\C\backslash C$$
  be the
fundamental group  of the fiber $\{f=t\}$ (it was denoted $F_X$ in \S \ref{mainsection}).
Consider  the trivial fiber bundle $\cup_{t\in\C\backslash C} \gr^k F_t\otimes_\Z \C$ and its  dual
  $\cup_{t\in\C\backslash C} \check{\gr}^k F_t\otimes_\Z \C$. Both of them have a canonical flat connexion
 defined as follows:

 Let $X=\{\omega_i\}_{i\in I}$, $\k=\C(t)$ and $\K$ the field of locally analytic
 multi-valued functions on $\C\backslash C$.
  The Gauss-Manin connection on ${\rm Ass}_X^1$ extends  canonically
 to a derivation operator
$$
\frac{\partial}{\partial t}=': {\rm Ass}_X\to \frac{1}{\Delta}{\rm Ass}_X
$$
which respects both the graduation of
${\rm Ass}_X$, the direct sum decomposition
${\rm Ass}_X=L_X\oplus S_X$, and satisfies the Leibniz rule  $(ab)'=a'b+ab',\ a,b\in {\rm Ass}_X$.
In particular for monomials in ${\rm Ass}_X$ it is given by:
$$
\omega_{i_1}\omega_{i_2}\cdots \omega_{i_r}\rightarrow
\sum_{j=1}^r\omega_{i_1}\omega_{i_2}\cdots\omega_{i_{j-1}}\omega_{i_j}'\omega_{i_{j+1}}\cdots\omega_r.
$$
The main property of the map $'$ is the following:
\begin{equation}
 \frac{\partial}{\partial t}\int_{\delta(t)}\omega=\int_{\delta(t)}\omega'
\end{equation}
for all $\omega\in {\rm Ass}_X^k$ and continuous family of closed loops $\delta(t)\in  F_t^k$.
The induced map
 $':L_X\to \frac{1}{\Delta}L_X$, where $L_X$ is the $\k$-vector space  $L_X\cong \frac{{\rm Ass}_X}{S_X}$, is the desired generalization
of the usual (algebraic) Gauss-Manin connexion.

%
\section{Non-persistence of Hamiltonian identity cycles}
\label{apandmo}
Let  $\mathcal{F}$ be a holomorphic foliation  with singularities on the plane $\C^2$.
 A cycle on a leaf of $\mathcal{F}$ is a non-zero free homotopy class of closed loops on this leaf.
 Let $\delta$ be a closed non-contractible loop contained in a  leaf and
denote the holonomy map associated to
$\delta$ by $h_\delta$.
The cycle represented by $\delta$ is said to be an  \emph{identity cycle} provided that $h_\delta$ is the
identity map, and \emph{Hamiltonian identity cycle} provided that $\mathcal{F}=\{ df=0\}$ for some polynomial
$f$. It is called a \emph{limit cycle} provided that $\delta$ corresponds to an isolated fixed point of $h_\delta $. Let
$\mathcal{F}^n$ be a polynomial foliation as above of degree $n$, and let $\delta$  be an identity cycle.

\textbf{Question:} \emph{Is there a sufficiently small degree $n$ perturbation of $\mathcal{F}^n$ which either destroys the cycle  $\delta$, or makes it a limit cycle ?}

In the case when $\mathcal{F}^n$ is a plane Hamiltonian foliation
the following theorem has been proved by Ilyashenko and Pyartli:
\begin{theo}
\label{maingeneric}( \cite[Theorem 2]{ilpy})
 For any Hamiltonian identity cycle $\delta$ in a given leaf of $\mathcal{F}^n= \{ df = 0 \},\ \ \deg(f)=n+1>3$, there exists a perturbation of this foliation in
the class of polynomial foliations of degree $n$ which destroys the identity  cycle .
\end{theo}
Our approach to the above problem of destroying the Hamiltonian identity cycles is as follows. Let $f$ be a
degree $n+1$ polynomial  and let $\omega=Pdx + Q dy$ be a polynomial one-form in $\C^2$ of degree at most $n$,
$\deg P, \deg Q \leq n$. Consider the holomorphic foliation on  $\C^2$ defined by
\begin{equation}
\label{perturbed}
df+\epsilon \omega=0 , \epsilon \sim 0 .
\end{equation}
Let $t$ be a regular value of $f$ and $\delta_t\subset f^{-1}(t)$ be a continuous family of   non-contractible
closed loops. As in the preceding sections we denote by $F_t$ the fundamental group of the fiber
$\{f^{-1}(t)\}$. There is  an integer $k\geq 1$ such that $\delta(t)$ represents a not zero element in $\gr^k
F_t$. We consider
 the holonomy map $h_\delta^\epsilon$ along the path $\delta(t)$
\begin{equation}
\label{holonomy}
h_\delta^\epsilon(t)=t+\sum_{i=1}^\infty \epsilon^i M_i(t).
\end{equation}
It is known \cite{ga06,ho06-5} that
 $M_1=\cdots=M_{k-1}=0$ and $M_k$ is an iterated path integral
\begin{equation}
\label{baazdele}
M_k(t)=\int_{\delta(t)}\omega ( \underbrace{ \omega( \cdots (\omega   ( \omega)'}_{k-1 \hbox{ times }} )'\cdots)')'
\end{equation}
 where $\omega'$ is defined by (\ref{18}). It follows from \S \ref{noandde} that the iterated integral in
 (\ref{baazdele}) is well defined, as it depends only on the classes of $\omega^{(i)},\ i=1,2,\ldots,k-1$ in the global Brieskorn module
 $H$ .

 If the function $M_k$ is not identically zero then  its zeros correspond to limit cycles of the deformed foliation.
 In particular
 the deformation
$df+\epsilon \omega$  " destroys" the family of identity cycles $\delta(t)$ in the sense that the holonomy map
$h_\epsilon, \epsilon \neq 0$, is not the identity map. This leads to the following:

\textbf{Question:}\emph{ Let $\delta(t)$ be a continuous family of closed loops representing a non-trivial
element of $ \gr^k F_t$. Is there a sufficiently small degree $n$ perturbation of $\{ df = 0\}$, such that the
corresponding Poincar\'e-Pontryagin function $M_k$ is not zero ?}

A positive answer to this question implies that the holonomy map (\ref{holonomy}) is not the identity map. We
shall prove that:
\begin{description}
    \item[(i)] The answer to the above question is, in general, negative (Example \ref{ex}).
    \item[(ii)] If the loop $\delta(t)$ is generic, then the answer is positive (Theorem  \ref{th4}).
    \item[(iii)] If the loop is a commutator, $\delta(t)\in \gr^2 F_t$, and the polynomial $f$ is of degree three
    (case not covered by \cite{ilpy}), then the answer is positive (Theorem \ref{elliptic}).
\end{description}
For completeness we consider first the seemingly trivial case $k=1$ (this is well known to the specialists). We
have:
\begin{prop}
\label{M1}
Let $f$ be a  degree $n+1>1$ polynomial, whose degree $n+1$ homogeneous part is non-degenerate
(has an isolated critical point). For every fixed non-critical value $t\in \C$ and for every closed loop
$\delta \subset f^{-1}(t)$ representing a non-zero homology class in $H_0(f^{-1}(t),\Z)$, there is a polynomial
one-form $\omega=Pdx+Qdy$ of degree at most $n$, $\deg P, \deg Q \leq n$, such that
$$
M_1'(t) = \int_\delta \omega' \neq 0 .
$$
\end{prop}
The above proposition implies that $M_1(t)$ is not identically zero and hence the cycle $\delta$ is destroyed by
the perturbation (\ref{perturbed}). This also follows (with a \emph{Morse-plus} condition on $f$) from the following
theorem, proved by Ilyashenko:
\begin{theo}
 [\cite{ilyashenko69}]
Let $f$ be a Morse-plus bi-variate polynomial, $\delta(t)\subset H_1(f^{-1}(t),\Z)$  a continuous family of
vanishing cycles, and $\omega$ a polynomial one-form of degree at most $n$ as above. Then $M_1(t) =
\int_{\delta(t)} \omega$ is  identically zero, if and only if $\omega  $ is exact.
\end{theo}

The condition that $f$ must be Morse-plus (has $n^2$ distinct critical values) can not be removed from
Ilyashenko's theorem.
\begin{proof}[Proof of Proposition \ref{M1}]Consider the one-forms $\omega_i'=
\frac{d\omega_i}{df}$ defined by (\ref{bases}). It is well known (and easy to check) that  when
 $|i|=i_1+i_2<n-1$, the $n(n-1)/2$ forms $\omega_i'$ generate the space of holomorphic differential one-forms on the compact Riemann surface $\Gamma=\overline{f^{-1}(t) }$
 $$H^0(\Gamma,\Omega^1)= \span\{\omega_i' : |i|=i_1+i_2 < n-1 \}
 $$
 while the $n(n+1)/2$ one-forms $\omega_i'$, $|i|=i_1+i_2 \leq n-1$, generate the space of meromorphic one-forms with at most a simple pole at infinity:
\begin{equation}
\label{h0}
 H^0(\Gamma,\Omega^1(\infty_1+\infty_2+\dots \infty_{n+1}))=  \span\{\omega_i': |i|=i_1+i_2 \leq n-1 \},
 \end{equation}
where $\infty_i$ are the $n+1$ distinct intersection points of $\Gamma$ with the line at infinity. This combined
to the Hodge decomposition
$$
H^1_{DR}(\Gamma)= H^0(\Gamma,\Omega^1) \oplus H^0(\Gamma,\bar{\Omega}^1)
$$
and the obvious equalities
$$
\dim H^0(\Gamma,\Omega^1(\infty_1+\infty_2+\dots \infty_{n+1}))= \dim H^0(\Gamma,\Omega^1) + n,
$$
$$
\dim H^1_{DR}(f^{-1}(t)) = \dim H^1_{DR}(\Gamma) +n
$$
implies
$$
H^1_{DR}(f^{-1}(t))= \span\{\omega_i': |i|=i_1+i_2 \leq n-1 \} \oplus \span\{\overline{\omega_i'} :  |i|=i_1+i_2 < n-1 \}.
$$
Let, for a fixed regular value $t$, $\delta \subset f^{-1}(t)$ be a closed loop representing a non-zero homology class in $H_0(f^{-1}(t),\Z)$.
If $\int_\delta \omega_i'=0, \forall i, |i|=i_1+i_2 \leq n-1$, then $\int_\delta \overline{\omega_i'} =0, \forall i, |i|=i_1+i_2 < n-1$ and hence the homology class of $\delta$ is zero. Proposition \ref{M1} is proved.
\end{proof}
Let, more generally, $\delta(t) \subset f^{-1}(t)$ be a continuous family of closed loops  representing a not zero element in $\gr^kF_t$, $k \geq 1$. Let
$V_n$ be the vector space of all pairs $(f,\omega) $, where
$f$ is a  bi-variate polynomial of degree at most $n+1$
and $\omega$ is a one-form of degree at most $n$.

\begin{theo}
\label{th4}
There is a dense open subset of $V_n$ with the following property:
for every pair $(f,\omega)$ of this open subset and for every integer $k\geq 1$ there exists a continuous family of
closed loops $\delta(t) \subset f^{-1}(t)$ representing a non-zero homology class in $\gr^kF_t$,   such that the corresponding
Poincar\'e-Pontryagin function $M_k(t)$ is not identically zero.
\end{theo}
The following example shows that the  result of the above theorem can not be improved.
\begin{exam}\rm
\label{ex} Let $f$ be a  polynomial of degree  $n+1$,  with $n^2$ distinct critical values, where $n\geq 2$
(such polynomials form a Zarisky open set in the space of degree $n+1$ polynomials). The generic leaf $
\{f=t\}\subset \C^2$ is a compact Riemann surface with $n+1$ removed points. Let $\alpha_1(t),
\alpha_2(t)\subset \{f=t\}$ be two continuous families of loops which make one turn about two distinct removed
points on $ \{f=t\}$, and put $X= \{\alpha_1, \alpha_2\}$.
 The element
$$
 ((( \alpha_1, \alpha_2), \alpha_1), ( \alpha_1, \alpha_2) )
$$
is a basic commutator of degree $5$ \cite[p.23]{se64}, and hence represents a non-zero equivalence class in
$L^5_X=F_5/F_6$. Define the continuous family of loops
$$
\gamma(t) = ((( \alpha_1(t), \alpha_2(t)), \alpha_1(t)), ( \alpha_1(t), \alpha_2(t)) ).
$$
We claim that for every polynomial one-form of degree at most $n$, the Poincar\'e-Pontrygin function $M_5$ associated to the deformed foliation (\ref{perturbed}) and to the family of loops $\gamma(t)$ is identically zero. Indeed, according to (\ref{baazdele}), the function $M_5$ is a polynomial in $\int_{\alpha_i(t)} \omega^{(k)}$, $i=1,2$, $k=0,1,...,4$. However, for $k\geq 2$ the one-form $\omega^{(k)}$ has no residues. Therefore (\ref{baazdele}) takes the form
$$
M_k(t)= \int_{\gamma(t)} \omega \omega'\omega'\omega'\omega'
$$
and hence
\begin{eqnarray*}
M_5(t)&=& \int_{(( \alpha_1(t), \alpha_2(t)), \alpha_1(t))}\omega \omega'\omega' \int_{( \alpha_1(t), \alpha_2(t))} \omega'\omega'\\
&   &- \int_{( \alpha_1(t), \alpha_2(t))} \omega \omega' \int_{(( \alpha_1(t), \alpha_2(t)), \alpha_1(t))}\omega'\omega'\omega' \\
& \equiv & 0 .
\end{eqnarray*}
\end{exam}
\begin{proof}[Proof of Theorem \ref{th4}]
Let $f^0$ be a degree $n+1$ polynomial, non-degenerate at infinity, and $\omega^0$ a degree $n$ one-form.
Denote by $\delta(t) \subset \{  {f}^0= t \}$  a continuous family of closed loops  representing a not zero element in $\gr^kF_t$, $k \geq 1$. The  function $M_k(t)$, associated to the pair $(f,\omega)\in V_n$ and the family of loops $\delta(t) \subset f^{-1}(t)$ is well defined, at least for
$(f,\omega)$ sufficiently close to
$(f^0,\omega^0)$. Therefore the condition that $M_k(t)$ is identically zero defines a closed subset of $V_n$ (possibly equal to $V_n$). It remains to show that there exists a point $(f^0,\omega^0)\in V_n$ at which the corresponding function $M_k(t)$ is not identically zero. Let $f^0$ be a homogeneous, degree $n+1$ polynomial, which is non-degenerate at infinity. We shall show that there exists a one-form $\omega$ of degree at most $n$, and a continuous family of closed loops  representing a not zero element in $\gr^kF_t$, such that the associated function $M_k(t)$ is not identically zero.

The proof is an application of the $\pi_1$ de Rham theorem formulated in \S \ref{mainsection}.
Let $X$ be the set of $n^2$ co-homology classes defined by the one-forms $\omega_i$  in
$H^1_{DR}(f^{-1}(t))$, see (\ref{bases}) .
By making use of the Picard-Fuchs system satisfied by $\omega_i$ we can represent the Poincar\'e-Pontryagin function $M_k(t)$,
(\ref{baazdele}),
in the form
\begin{equation}
M_k(t)= \int_{\delta(t)} P_k(\omega),
\end{equation}
where $P_k(\omega) \in {\rm Ass}_X^k$ is a suitable degree $k$ non-commutative  polynomial.
If the function $M_k(t)$ were identically zero for all family of cycles $\delta(t)$, then the polynomial  $P_k(\omega)$ would belong, for every regular value $t$, to the space $S_X^k$ generated by shuffle products.
Let $V$ be the vector space generated by $\omega_i$'s.
It is easy to check that the set $\{\omega\in V\mid  P_k(\omega)\in S_X^k\}$ is algebraic. We are going to prove that it is not the whole $V$.

Taking into consideration the non-degeneracy of $f$ at infinity, we deduce that the  Picard-Fuchs system associated to $\omega_i$ is diagonal
$$
t \omega_i' = w_i \omega_i ,\mbox{  where  } w_i={\rm weight}(\omega_i):=
\frac{|i|+2}{n+1} .
$$
This allows to deduce an explicit formula for $P_k(\omega)$. Let, for instance,
$\omega = \alpha_1 \omega_1 +  \alpha_2 \omega_2$, where ${\rm weight}(\omega_i)= w_i$. One easily proves by induction that
\begin{equation}
\label{pk}
P_k(\omega) = \frac{1}{t^{k-1}} \sum_{i=0}^k \alpha_1^i \alpha_2^{k-i} \sum c_{i_1i_2\dots i_k} \omega_{i_1}\omega_{i_2}\dots \omega_{i_k}
\end{equation}
where
$$c_{i_1i_2\dots i_k}=
w_{i_k} (w_{i_k}+ w_{i_{k-1}}-1)\dots (w_{i_k}+ w_{i_{k-1}} +\dots +w_{i_2}-k+2)
$$
and in the second sum of (\ref{pk}), the index $i_1i_2\dots i_k$ runs through all the shuffles of
$$
\underset{\mbox{$i$ times}}{\underbrace{(1\dots1)}} \mbox{   and   } \underset{\mbox{$k-i$ times}}{\underbrace{(2\dots2)}} .
$$
Note that the first sum in (\ref{pk}) is with respect to all partitions $i+ (k-i)$ of $k$. This fact reflects a further decomposition of $S_X^k$ and $L_X^k$ in direct orthogonal sums corresponding to partitions of $k$. The polynomial $P_k(\omega)$ belongs to $S^k_X$ if and only if it is orthogonal to $L_X^k$ (Theorem \ref{prop1}). It suffices to show that the coefficient of at least one monomial $\alpha_1^i \alpha_2^{k-i}$ in (\ref{pk}) is not orthogonal to $L_X^k$. Take for instance the partition $k=1+(k-1)$. The coefficient of $\alpha_1 \alpha_2^{k-1}$ takes the form
\begin{eqnarray*}
P_k^1(\omega_1,\omega_2) &= &\omega_1\omega_2^{k-1} w_2(2w_2-1)\dots ((k-1)w_2-k+2)\\
& +& \omega_2\omega_1\omega_2^{k-2} w_2(2w_2-1)\dots ((k-2)w_2-k+3) ((k-2)w_2+w_1-k+2)\\
& &\vdots  \\
& + &  \omega_2^{k-1} \omega_1 w_1 (w_1 + w_2-1)\dots ((k-2)w_2+w_1-k+2) .
\end{eqnarray*}
The subspace of $L^k_X$ corresponding to the partition $k=1+(k-1)$ is one-dimensional and generated by
the Lie polynomial
$$
L_k^1(\omega_1,\omega_2)= [[\dots[[\omega_1,\omega_2],\omega_2],\dots , \omega_2],\omega_2] .
$$
Consider the scalar product
\begin{equation}
\label{scalar}
C_k(w_1,w_2) = \langle P_k^1(\omega_1,\omega_2), L_k^1(\omega_1,\omega_2)\rangle   .
\end{equation}
The identities
\begin{eqnarray*}
\langle L_k^1, \omega_{i_1} \omega_{i_2}\dots \omega_{i_k}\rangle  &=&\langle [L_{k-1}^1,\omega_2], \omega_{i_1} \omega_{i_2}\dots \omega_{i_k}\rangle  \\
&=&\langle L_{k-1}^1,\omega_{i_1} \omega_{i_2}\dots \omega_{i_{k-1}}\rangle  \langle \omega_2, \omega_{i_k}\rangle  \\
&   & - \langle \omega_2, \omega_{i_1}\rangle   \langle L_{k-1}^1,\omega_{i_2} \omega_{i_3}\dots \omega_{i_{k}}\rangle
\end{eqnarray*}
applied to the scalar product $C_k$ lead to the
 following recursive formula$$
C_k(w_1,w_2)= (w_2-w_1) C_{k-1}(w_1+w_2-1,w_2), k\geq 3, \;\;C_2(w_1,w_2)= w_2-w_1
$$
and hence
$$
C_k(w_1,w_2)=  (w_2-w_1) \Pi_{i=1}^{k-2} (i-w_1-(i-1)w_2) .
$$
Recall now that $w_i:= \frac{|i|+2}{n+1}$, where $|i|+2 \leq n-1$.
If we choose $w_1\neq w_2$ and $w_1 < 1$ then the
scalar product $C_k(w_1,w_2)$ is not zero. The Theorem is proved.
\end{proof}

\begin{theo}
\label{elliptic} Let $f$ be a  degree three polynomial, whose degree three homogeneous part is non-degenerate
(has an isolated critical point). Let $\delta (t)\subset \{f=t\}$ be a continuous family of loops representing a
non-zero element in $F^2_X/F^3_X$, where $X$ is a basis of the first homology group $H_1(\{f=t\},\Z)$. Then
there exists a one-form $\omega$ of degree at most two, such that
$$
M_2(t)=\int_{\delta (t)} \omega \omega' \not \equiv 0 .
$$
\end{theo}
For a generic $f$ and a particular class of cycles $\delta(t)$ one can say that $M_2=0$ if and only if $\omega$ is exact. For more details on this see \cite{ho06-5}.
  Before proving the above theorem we need some preparation. For a generic $t$ the affine algebraic curve
$\{f=t\}$ is a topological torus with three removed points, and $f$ has four critical points. Let $h_i\in
{\rm Aut}(H_1(\{f=t_0\},\Z)),\ i=1,2,3,4$ be the associated monodromy operators ($t_0$ is a fixed regular value). Let
$\delta _i(t), \alpha_i(t)$, $i=1,2$, be continuous family of cycles in $H_1(\{f=t\},\Z)$ which form a basis. We
suppose moreover that $\delta _1(t), \delta _2(t)$ are vanishing cycles, while $\alpha_1(t), \alpha_2(t)$ are
homologous to zero on the compactified curve $\{f=t\}$. Without loss of generality we have
$h_i(\alpha_j)=\alpha_j$, $ h_i(\delta _i)=\delta _i$, and
$$
h_1(\delta _2)= \delta _2 + \delta _1, h_2(\delta _1)= \delta _1 -\delta _2 .
$$
Identifying $F^2_X/F^3_X$ to $L^2_X$, where
$$X= \{\delta _1, \delta _2, \alpha_1, \alpha_2 \}$$
we represent an element $\delta \in F^2_X/F^3_X$  as a linear combination of commutators
$$
[\delta _1,\delta _2], [\delta _i,\alpha_j], [\alpha_1,\alpha_2], \; \; i,j=1,2 .
$$
The monodromy operators $h_i$ induce automorphisms of $F^2_X/F^3_X$ denoted by the same letter.

\begin{prop}
\label{mon} Let $\delta (t)$ be a continuous family of loops representing a non-zero equivalence class in
$F^2_X/F^3_X$. There exists a polynomial $P$ in $h_i$ and with integer coefficients, such that
$$
P(h_1,h_2,h_3,h_4)(\delta )= k [\alpha_1,\alpha_2],
$$
where $k$ is a non-zero integer.
\end{prop}
\begin{proof}
 The proof is straightforward. Let $\delta _i, \ i=1,2,3,4$ be vanishing cycles associated to the critical values
of $f$. The Dynkin diagram of $f$ is of type $D_4$ and we may suppose that the intersection matrix is
 $$
(\delta _i\cdot \delta _j)=
\left(%
\begin{array}{cccc}
   0 & 1 & 0 & 0 \\
  -1 & 0 & 1 & 1 \\
  0 & -1 & 0 & 0 \\
  0 & -1 & 0 & 0 \\
\end{array}%
\right),
$$
$\alpha_1=\delta _1-\delta _3$ and $\alpha_2=\delta _1-\delta _4$ and
$$h_i(\delta ) = \delta  -
(\delta \cdot\delta _i) \delta _i $$ (the Picard-Lefschetz formula). Every $\delta \in F^2_X/F^3_X$ is of the form
\begin{equation}\label{gamma}
\delta = [\delta _1,a] + [\delta _2,b]+ m [\delta _1,\delta _2]+n[\alpha_1,\alpha_2]
\end{equation}
where $m,n\in \Z$ and $a,b$ are integer linear combinations in $\alpha_1, \alpha_2$.
 We have
\begin{eqnarray}
 \label{26} (h_1-id)(\delta ) &=& [\delta _1,b] \\
 \label{27} (h_2-id)(\delta ) &=& -[\delta _2,a]
\end{eqnarray}
which shows that it is enough to prove the Proposition for $\delta $ as in (\ref{gamma}) with $m=n=0$ (provided
that either $a$ or $b$ is non-zero). In the case $m=0$ we have
\begin{eqnarray}
 \label{28} (h_3-id)(\delta ) &=& -[\delta _3,b] \\
 \label{29} (h_4-id)(\delta ) &=& -[\delta _4,b]
\end{eqnarray}
 Therefore by (\ref{26}),  (\ref{27}), (\ref{28}) there is a polynomial $P$ (a linear combination in
$h_1,h_3,h_4,id$) such that $P(\delta )=[\alpha_1,b]$ or $P(\delta )=[\alpha_2,b]$. If $b\neq 0$ the Proposition
is proved. If $b=0$ but $a\neq 0$ then using  (\ref{27}) we replace without loss of generality $\delta $ by
$[\delta _2,a]$ and conclude as above. It remains the case $a=b=0$ but $m\neq 0$. We have
\begin{eqnarray}
 \label{30} (h_3-id)(\delta ) &=& -[\delta _1,\delta _3] \\
 \label{31} (h_4-id)(\delta ) &=& -[\delta _1,\delta _4]
\end{eqnarray}
and hence $(h_3-h_4)(\delta )= [\delta _1,\delta _4-\delta _3]= [\delta _1,\alpha_1-\alpha_2]$. We may therefore
replace $\delta $ by $[\delta _1,\alpha_1-\alpha_2]$ and conclude as above. Proposition \ref{mon} is proved.
\end{proof}
\begin{proof}[Proof of Theorem \ref{elliptic}.]
 If $M_2(t)$ were identically zero then each analytic
continuation of $M_2(t)$ would be zero too. By Proposition \ref{mon} we conclude that the function
$$
F_\omega(t)= \int_{[\alpha_1(t),\alpha_2(t)]}\omega \omega'=
det \left(%
\begin{array}{cc}
  \int_{\alpha_1(t)} \omega & \int_{\alpha_2(t) } \omega\\
  \int_{\alpha_1(t)} \omega' & \int_{\alpha_2(t)} \omega'
\end{array}%
\right)
$$
is identically zero. We shall show that that there is a polynomial one-form $\omega$ of degree at most two, such
that $F_\omega(t)\not\equiv 0$.   Indeed, the affine curve $\{f=t\}$ is a topological torus with three removed
points $\infty_1, \infty_2, \infty_3$ and the vector space $W$, (\ref{h0}), of meromorphic one-forms with at
most simple poles at $\infty_i$ is three dimensional and generated by
$$\omega_{00}'= \frac{dx}{f_y}, \omega_{10}'=2\frac{x\, dx}{f_y},\omega_{01}'= 2 \frac{y \, dx}{f_y}$$
(see the proof of Proposition \ref{M1} and \cite[Lemma 3]{Gav4}) where
$$
d \omega_{00}= dx \wedge dy,d \omega_{10}=x \, dx \wedge dy, d \omega_{01}= y \, dx \wedge dy
$$
and $\omega'$ is the Gelfand-Leray residue  of  $d\omega$. Furthermore, the integrals $\int_{\alpha_i(t)}
\omega$ are just the residues of $\omega$ at $\infty_i$. These residues are linear in $t$ when $\omega \in W$.
The identity $F_\omega(t)\equiv 0$ (for every fixed $\omega\in W$) and the linearity of $\int_{\alpha_i(t)}
\omega$ imply that the direction of the vector
$$
(\int_{\alpha_1(t)} \omega , \int_{\alpha_2(t) }\omega)
$$
does not depend on $\omega\in W$ neither on $t$ (when it is defined). This shows that a suitable linear
combination  of $\omega_{10}',\omega_{01}'$ with constant coefficients has no residues and hence is a
holomorphic one-form. Therefore the one-forms $\omega_{00}', \omega_{10}',\omega_{01}'$ are linearly dependent,
which is a contradiction. Theorem \ref{elliptic} is proved.
\end{proof}

\def\cprime{$'$} \def\cprime{$'$} \def\cprime{$'$}



\begin{thebibliography}{1}
\bibitem{gav}
L. Gavrilov,  Petrov modules and zeros of Abelian integrals. Bull. Sci. Math. 122 (1998), no. 8, 571--584.
\bibitem{Gav4} L. Gavrilov, The infinitesimal 16th Hilbert problem in the quadratic case, {\it Invent. Math.}
{\bf 143} (2001), 449--497.
\bibitem{gav03} L. Gavrilov, Talk on the conference
"Hilbert's Sixteenth and Related Problems in
Dynamics, Geometry and Analysis"
In honor of the 60-th anniversary of Yulij Sergeevich Ilyashenko, Moscow, 2003.
\bibitem{ga06}
L.~Gavrilov,
\newblock Higher order Poincar\'e-Pontryagyn functions and iterated path
  integrals.
\newblock {\em   Ann. Fac. Sci. Toulouse Math.} (6)  14,  no. 4, 663--682, 2005.
\bibitem{gi} L. Gavrilov, I.D. Iliev, The displacement map associated to polynomial unfoldings of planar
Hamiltonian vector fields, American J. of Math., 127 (2005) 1153-1190.
\bibitem{hain86}
R. Hain,  On the indecomposable elements of the bar construction. Proc. AMS 98
(1986), 312--316.
\bibitem{hain1}
R. Hain, Iterated integrals and algebraic cycles: examples and prospects.
Contemporary trends in algebraic geometry and algebraic topology (Tianjin,
2000), 55--118, Nankai Tracts Math., 5, World Sci. Publ., River Edge, NJ, 2002.
\bibitem{hain} R. Hain, The geometry of mixed Hodge structure on the fundamental group, Proc. of Simposia in
Pure Math., vol. 46 (1997) 247-282.
\bibitem{ha59}
Marshall Hall, Jr.
\newblock {\em The theory of groups}.
\newblock The Macmillan Co., New York, N.Y., 1959.
\bibitem{ilyashenko69} Yu.S. Ilyashenko, Appearance of limit cycles under
perturbation of the equation $dw/dz=-R_x/R_w$ where $R(z,w)$ is a
polynomial, {\it Mat. Sbornik} {\bf 78} (1969), 260--273. [in Russian]
\bibitem{ilpy}
Yu.~S.~ Ilyashenko and A.~S.~Pyartli,
\newblock The monodromy group at infinity of a generic polynomial vector field on
the complex projective plane,
\newblock {\em Russ. J. Math, Phy.}, Vol. 2, No. 3, 275--3155, 1994.
\bibitem{il96} Comments of Yu.~S.~ Ilyashenko to the abridged version of \cite{pe55},  published
in \\
Petrowsky, I. G. Selected works. Part II. Classics of Soviet Mathematics, 5.
Gordon and Breach Publishers, Amsterdam, 1996.
\bibitem{ilpreface} Yu.S. Ilyashenko, Preface to
Proc. of the Steklov Institute of Mathematics, 2006, Vol. 254, pp. 1-6.

\bibitem{ilyashenko06} Yu.~S.~ Ilyashenko, Persistence Theorems and Simultaneous Uniformization,
Proc. of the Steklov Institute of Mathematics, 2006, Vol. 254, pp. 184-200.

\bibitem{lambe}
L.~Lambe, Notes on symbolic computations, {\tt www.matematik.su.se/$\sim$lambe/}
\bibitem{ho06-5}
H.~Movasati and I.~Nakai,
\newblock {Commuting holonomies and rigidity of holomorphic foliations}
\newblock {\em  Bull. London Math. Soc.}, 40(2008), 473-478.

\bibitem{ho06-6}
H.~Movasati,
\newblock Mixed Hodge structure of affine hypersurfaces.
\newblock {\em Ann. Inst. Four},  57 no. 3 (2007), p. 775-801.

\bibitem{nakai}
I.~Nakai and K.~Yanai,
\newblock Relations of formal diffeomorphisms.
\newblock {\em RIMS, Kokyuroku}, 1447, 2005.

\bibitem{pe55}
I. G. Petrovskii and E. M. Landis,
\newblock {On the Number of Limit Cycles of the Equation dy/dx = P (x, y)/Q(x, y),
where P and Q are Polynomials of the Second Degree,} \newblock {\em Mat. Sb. 37 (2), 209\"{\i}?`½250 (1955) [AMS Transl., Ser. 2,
10, 177\"{\i}?`½221 (1958)]}.
\bibitem{reutenauer}
C. Reutenauer,  Free Lie algebras, London Mathematical Society Monographs.
Oxford University Press, New York, 1993.



\bibitem{re58}
Rimhak Ree,
\newblock Lie elements and an algebra associated with shuffles.
\newblock {\em Ann. of Math. (2)}, 68:210--220, 1958.
\bibitem{se64}
Jean-Pierre Serre,
\newblock {\em Lie algebras and {L}ie groups}, volume 1500 of {\em Lecture
  Notes in Mathematics}.
\newblock Springer-Verlag, Berlin, 2006.
\newblock 1964 lectures given at Harvard University, Corrected fifth printing
  of the second (1992) edition.
\bibitem{wi53}
Ernst  Witt,
Treue Darstellungen beliebiger Liescher Ringe.
Collectanea Math. 6, (1953). 107--114.

\end{thebibliography}
\end{document}